\documentclass[a4paper,12pt]{amsart}
\usepackage[utf8]{inputenc}
\usepackage{setspace}
\usepackage{amsmath,amsthm,amsfonts}
\usepackage[T1]{fontenc}
\usepackage{graphicx}
\usepackage{amssymb}
\usepackage[all,cmtip]{xy}
\usepackage{tikz}
\usepackage{indentfirst}
\usepackage[mathscr]{euscript}
\usepackage{mathrsfs}

\usetikzlibrary{matrix,arrows,decorations.pathmorphing}

\newtheorem{thm}{Theorem}[section]

\newtheorem{lem}[thm]{Lemma}
\newtheorem{prop}[thm]{Proposition}
\theoremstyle{definition}

\theoremstyle{remark}

\newtheorem{example}{Example}[section]

\newcommand{\Z}{\mathbb Z}
\newcommand{\C}{\mathbb C}

\title[Perverse Sheaves on Plane Line Arrangements]{Decomposition of Perverse Sheaves on Plane Line Arrangements}

 \author[R.~B\o gvad]{Rikard B\o gvad}
\address{Department of Mathematics, Stockholm University, SE-106 91
Stockholm,         Sweden}
\email {rikard@math.su.se }

\author[I.~Gon\c calves]{Iara Gon\c calves}
\address{Department of Mathematics, Stockholm University, SE-106 91
Stockholm,         Sweden}
\email{iara@math.su.se}

\begin{document}

\begin{abstract}
On the complement $X= \C^2 - \bigcup_{i=1}^n L_i$ to a central plane line arrangement $\bigcup_{i=1}^n L_i \subset \C^2$, a locally constant sheaf of complex 
vector spaces $\mathcal L_a$ is associated to any multi-index $a \in \C^n$. Using the description of MacPherson and Vilonen of the category of perverse sheaves (\cite{MV2} and \cite {MV3}) we obtain a criterion for the irreducibility and number of 
decomposition factors of the direct image $Rj_* \mathcal L_a$ as a perverse sheaf, where $j: X \rightarrow \C^2$ is the canonical inclusion.
\end{abstract}

\maketitle

\section{Introduction}

Recall that the category of perverse sheaves is an 
abelian category where each object has a finite decomposition series. The aim of this work is to analyze the irreducible factors in a decomposition series of a direct image of a rank 1 locally constant sheaf, $\mathcal L_a$. 
To define $\mathcal L_a$, consider a central line arrangement $S:=\bigcup_{i=1}^n L_i \subset \C^2$. Denote by $j$ the affine inclusion $U:=\C^2\setminus S\to \C^2$. 
A locally constant sheaf (of complex vector spaces) on $U$ is given by a complex representation of $\pi_1(U)$, a group that is generated by 
loops around $L_i$. Hence, to each multi-index $a =(a_1, \ldots, a_n)\in \C^n$, there is associated a rank 1 sheaf $\mathcal L_{a}$, where $a_i$ is the result of the action of a loop around $L_i$.
Our main result is a description of the number of decomposition factors of the perverse sheaf $Rj_* \mathcal L_a$ in terms of $a$; in particular we give a criterion for irreducibility. 

\begin{thm}
\label{j2irre}
The perverse sheaf $Rj_* \mathcal L_{a}$, where $j: \C^2 - \cup_{i=1}^n L_i \rightarrow \C^2$, is irreducible if, and only if, both of the following conditions are satisfied:
\begin{itemize}
\item $a_i \neq 1$, for all $i=1, \ldots, n$;
\item $\Pi_{i=1}^{n} a_i \neq 1$.
\end{itemize}
\end{thm}

The proof of this theorem and the related results on the number of decomposition facors, is contained in section 3. It proceeds by using the quiver methods developed by MacPherson and Vilonen in \cite{MV2}, and \cite {MV3}; these are recalled and adapted to our case in section 2. The proof, ending in an intricate determinant computation, gives i.a.  a quiver description of the irreducible factors of $Rj_* \mathcal L_a$, that may also be of interest.
By the Riemann-Hilbert correspondence, the theorem relates to a similar assertion of D-modules, which was treated by Abebaw-B\o gvad in \cite{AB}, Theorem 1.3. Our original motivation for this project was the hope that the proofs of that paper, which use totally different techniques, would be simplified in the perverse setting. A secondary motivation was to provide an explicit accessible example of computations with the quiver description given in the above references.

A generalization of the irreducibility conditions to arbitrary 
arrangements was presented by Budur et al. in the recent \cite{Bud}.

\subsection{Introductory example}

A basic example of what we intend to do for plane arrangements is the following (wellknown) result.
Let $X$ be a path-connected topological space and $\pi_1(X, x_0)$ be its fundamental group with base point $x_0$. The monodromy representation of $\pi_1(X,x_0)$ on the stalk $\mathcal L_{x_0}$, defines an equivalence 
between the category of locally constant sheaves $\mathcal L$ on a space $X$ and the category of finite dimensional complex representations of the fundamental group of $X$.

Consider $\C^* = \C - \{ 0 \}$ with the fundamental group $\pi_1(\C^*)= \{ n \gamma, n \in \Z \} \cong \Z$, where $\gamma$ is a loop going around $0$ once starting in $x_0=1$. A locally 
constant sheaf $\mathcal L$ of rank 1 on $\C^*$ is classified, up to isomorphism, by the element $\alpha \in \C$, such that for the monodromy representation $\gamma {\bf e} = \alpha {\bf e}$, 
where $ \C {\bf e}$ is the stalk of $\mathcal L$ at $1 \in \C^*$.
Denote by $\mathcal L_{\alpha}$ the rank 1 sheaf corresponding to $\alpha \in \C$.
Consider the inclusion $j: \C^* \rightarrow \C$. Denote by $\mathcal Sk$ the skyscraper sheaf at $0\in \C$, defined by having its only non-trivial stalk equal $\C$ at $0\in \C$.

\begin{prop} 
\label{prop:Clc}
\begin{enumerate}
\item If $\alpha\neq 1$ then $$j_!\mathcal L_\alpha\cong R^0 j_*\mathcal L_\alpha\cong R j_*\mathcal L_\alpha$$
are isomorphic and irreducible perverse sheaves in the derived category of sheaves $D_{\C}$ in $\C$.
\item If $\alpha= 1$, so that $\mathcal L_1$ is a constant sheaf on $\C^*$, there are short exact sequences of perverse sheaves
$$\mathcal Sk[-1]\hookrightarrow j_!\mathcal L_1 \twoheadrightarrow  R^0 j_*\mathcal L_1$$
and
$$R^0 j_*\mathcal L_1\hookrightarrow R j_*\mathcal L_1 \twoheadrightarrow \mathcal Sk[-1].$$
Furthermore $\mathcal Sk[-1]$ and $R^0 j_*\mathcal L_1$ are irreducible perverse sheaves.
\end{enumerate}
\end{prop}

For more details and a proof, see Gon\c calves \cite{Gon}; in particular the proof given there just uses the definition of perverse sheaves.\vspace{0.2cm}

Our main reference on perverse sheaves is the work by Beilinson, Bernstein and Deligne, \cite{BBD}. We have also used the papers of Massey, \cite{DM}, and Rietsch, \cite{KoR}. Throughout we only consider the middle perversity. In contrast to \cite{BBD} we use the convention that a locally constant sheaf placed in homological degree $0$ is perverse. This amounts to a shift of the derived category by $\text{dim }X$.

It is due to P. Deligne that we have a quiver description of irreducible perverse sheaves. This idea was developed and discussed in more detail by R. MacPherson and K. Vilonen in \cite{MV1}, \cite{MV2} 
and \cite{MV3}. Our next section is based on these last papers.

\section{Preliminaries}

We describe, following  \cite{MV2}, a category $\mathcal C (F,G;T)$ that is equivalent to the category of perverse sheaves, $\mathcal M(X)$ on a stratified space $X$. The idea is roughly that the category of perverse sheaves is glued from the category of 
perverse sheaves $\mathcal A$ on an open strata of maximal dimension, and the category of perverse sheaves $\mathcal B$ on the complement.
\subsection{Categorical construction}

The initial definitions as stated in an abstract setting in \cite{MV2} are as follows: consider two categories $\mathcal A$ and $\mathcal B$, two functors $F$ and $G$ from $\mathcal A$ to $\mathcal B$, and a 
natural transformation $T$ from $F$ to $G$. Symbolically $F,G:\mathcal A \rightarrow \mathcal B$ and $F \xrightarrow{T} G$. Define the category $\mathcal{C}(F,G;T)$ to be the category whose objects are pairs 
$(A, B) \in Obj\mathcal{A} \times Obj \mathcal{B}$ together with a commutative triangle
\begin{equation}
\label{eq:diagram}
\xymatrix{
FA \ar[rr]^{T_A} \ar[rd]_{m} & & GA \\ 
& B \ar[ur]_n &  }
\end{equation}
and whose morphisms are pairs $(a,b) \in Mor \mathcal A \times Mor \mathcal B$ such that
$$\xymatrix{
FA \ar[rr]^{T_A} \ar[rd]^m \ar[dd]_{Fa} & & GA \ar[dd]^{Ga} \\ 
& B \ar[ur]^n \ar[dd]_<<<<<<b &  \\
FA' \ar'[r][rr]^{T_{A'}} \ar[rd]_{m'} & & GA' \\
& B' \ar[ur]_{n'} & }
$$

\begin{prop} (\cite{MV2}, pp. 405)
If $\mathcal A$ and $\mathcal B$ are abelian categories and if $F$ is right exact and if $G$ is left exact, then the 
category $\mathcal C(F,G;T)$ is abelian and the functors taking $(A,B) \mapsto A$ and $(A,B) \mapsto B$ from 
$\mathcal C(F,G;T)$ to $\mathcal A$ and $\mathcal B$ are exact.
\end{prop}

In this context we have the following canonical functors from $\mathcal A$ to $\mathcal C (F,G;T)$, $\widehat{F}$, $\widehat{G}$ and $\widehat{T}$ given, respectively, by:

\noindent
\begin{center}
\begin{equation}
\label{tilt}
\widehat{F}(A):=\xymatrix{
FA \ar[rr]^{T_A} \ar[rd]_{id} & & GA\\
& FA \ar[ru]_{T_A} &
}\hspace{1.0cm}\widehat{G}(A):=\xymatrix{
FA \ar[rr]^{T_A} \ar[rd]_{T_A} & & GA\\
& GA \ar[ru]_{id} &
}
\end{equation}
\end{center}
$$\widehat{T}(A):=\xymatrix{
FA \ar[rr]^{T_A} \ar@{->>}[rd] & & GA\\
& Im T_A \ar[ru] &
}$$
It is easy to characterize the irreducible objects in $\mathcal C(F,G;T)$.
\begin{prop}(\cite{KV}, pp.666)
\label{irpor}
Irreducible objects in $\mathcal C(F,G;T)$ are either of the form $\widehat{T}(L)$, where $L$ is an irreducible object
in $\mathcal A$, or of the form
$$\xymatrix{
0 \ar[rr] \ar[rd] & & 0 \\
 & L \ar[ru] &
}
$$
where $L$ is irreducible in $\mathcal B$. Conversely, all such objects are irreducible.
\end{prop}

\subsection{Application to perverse sheaves}

Let $X$ be a complex manifold with a given Whitney stratification $\mathcal S$. We now sketch how the above construction applies to the category  $\mathcal M(X)$ of perverse sheaves with respect to $\mathcal S$(for details, see \cite{MV2}, Section 5).

Let $S$ be a closed stratum of complex dimension $d$. Denote
$$\Lambda_S = T_S^*X \qquad \mbox{and} \qquad \tilde{\Lambda}_S = \Lambda_S - \bigcup_{R \neq S} \overline{T_R^*X},$$
where $T_S^*X$ is the cotangent bundle of $X$ at $S$ and the union is over the other strata $R$.
We have:
\begin{itemize}
\item the right exact functor $\psi: \mathcal M(X -S) \rightarrow  \{ \mbox{local systems in} \ \tilde{\Lambda}_S \}$, called the nearby cycles;
\item the left exact functor $\psi_c: \mathcal M(X -S) \rightarrow  \{ \mbox{local systems in} \ \tilde{\Lambda}_S \}$, called the nearby cycles with compact support;
\item a functor $\Phi: \mathcal M(X) \rightarrow  \{ \mbox{local systems in} \ \tilde{\Lambda}_S \}$, called the vanishing cycles;
\item a natural transformation $var: \psi \rightarrow \psi_c$, called the variation.
\end{itemize}

We will later see how these functors are calculated in our plane case.

\vspace{0.2cm}

From this we can build a category $C:=\mathcal C (\psi, \psi_c, var )$, where the categories earlier denoted by $\mathcal A$ and $\mathcal B$ are, respectively, the category of perverse sheaves on $X- S$ and the 
category of local systems in $\tilde{\Lambda}_S$. 

Starting with a perverse sheaf $P^{\bullet}\in M(X)$, there is an associated diagram (\ref{eq:diagram}) in $C$, with $A=j^*P^{\bullet}$, $B=\Phi(P^{\bullet})$.
This actually gives an embedding of $M(X)$ in $C$:

\begin{prop} (\cite[Thm 5.3]{MV2}) $M(X)$ is equivalent to a subcategory $\tilde M(X)$ of $\mathcal C (\psi, \psi_c, var )$. In addition,  $\tilde M(X)$ is closed under taking subobjects.
\end{prop}
\begin{proof}The only statement not in \cite[Thm 5.3]{MV2}, is the last statement, which follows from the observation that the condition that describes $\tilde M(X)$ is 
invariant with respect to taking subobjects. Let $C: \psi A\to \Phi A\to \psi_c A$ (with maps $m$ and $n$, respectively) be an object of of $\mathcal C (\psi, \psi_c, var )$ that corresponds to a perverse sheaf. Given $\alpha\in \pi_1(p^{-1}x)$ where $p$ is the projection $ \tilde{\Lambda}_S\to S$ there is defined a natural transformation $I_\alpha: \psi_c\to\psi$, and the condition is that for the monodromy action $
\mu_\alpha
$ on $\Phi A$ the identity \begin{equation}
\label{eq:condition}\mu_\alpha-1=m\circ I_\alpha\circ n\end{equation} holds. Now it is easy to see that for a subobject $\psi A'\to \Phi A'\to \psi_c A'$ of $C$ the map $\Phi A'\to \Phi A$ is an injection. Hence if (\ref{eq:condition}) holds for $\Phi A$ it will hold 
for $\Phi A'$.

\end{proof}

We have three functorial ways of extending a perverse sheaf $P^{\bullet} \in \mathcal M(X -S)$ to an object in $\mathcal M(X)$.
The corresponding element in  the category $\mathcal C(F,G;T)$ can be very nicely described.
Let $F = \psi$, $G = \psi_c$ and $T = var$.
 
\begin{prop}(see \cite{MV2}, pp.418)
\label{3functors}
The functors $^pj_!P^{\bullet}$, $^pj_{!*}P^{\bullet}$ and $^pj_*P^{\bullet}$ correspond, respectively, to the functors
$\widehat{F}$,$\widehat{T}$ and $\widehat{G}$ (defined in \eqref{tilt}).  That is, they are given as the following factorizations of 
$T: F P^{\bullet} \rightarrow G P^{\bullet}$:
\begin{itemize}
\item $^pj_!P^{\bullet}$: $F P^{\bullet} \xrightarrow{id} F P^{\bullet} \xrightarrow{T} G P^{\bullet}$;
\item $^pj_{!*}P^{\bullet}$: $F P^{\bullet} \twoheadrightarrow im(T) \hookrightarrow G P^{\bullet}$;
\item $^pj_*P^{\bullet}$: $F P^{\bullet} \xrightarrow{T} G P^{\bullet} \xrightarrow{id} G P^{\bullet}$.
\end{itemize}
\end {prop}

\subsection{The category $ \mathcal C(\psi, \psi_c, var)$ in the plane case}
\label{section:plane1}
MacPherson and Vilonen \cite{MV3}  applied the previous ideas to the plane curve $S: y^m = x^n$. We will take $m=n$, so that $S$ corresponds to a central arrangement 
of $n$ lines $S_i = \{  x - \epsilon_n^i y = 0  \},\ i=1,...,n$ (where $\epsilon_n$ is a $n$:th primitive root of unity).

Consider the space $\C^2$ and the stratification $\{ 0 \} \subset S \subset \C^2$, where $S = \bigcup_{i=1}^{n} S_i$. This is a Whitney stratification.
We will first give a description of $\mathcal M(\C^2 - \{ 0 \})$. 
Define $\Lambda = T^*_{\C^2}\C^2 \cup \overline{T^*_S\C^2}- \{ 0 \}  \cup T^*_{ \{ 0 \}}\C^2$.

The standard hermitian metric allows us to identify $T^*_{S_i}\C^2$ (the cotangent bundle) with $T_{S_i}\C^2$ (the tangent bundle). Let $U_i$ be a neighborhood of the zero-section of $T^*_{S_i}\C^2$ and 
$V_i$ a neighborhood of $S_i$ in $\C^2$.
By the tubular neighborhood theorem there exists a diffeomorphism between $U_i$ and $V_i$. 
If $A$ is a local system in $\C^2 - S$, for each $\tilde{\Lambda}_{S_i}$, we can define a local system $\tilde{A_i}$ obtained by the pullback of the following maps between fundamental groups:
$$\pi_1(\tilde{\Lambda}_{S_i}) \xleftarrow{\cong} \pi_1(U_i - S_i) \xleftarrow{\cong} \pi_1(V_i - S_i) \rightarrow \pi_1(\C^2 - S).$$
Noticing that $\tilde{\Lambda}_{S_i}= S_i \times \C^*$ and considering the projection $\pi: \tilde{\Lambda}_{S_i} \rightarrow S_i$, we observe that $\tilde{\Lambda}_{S_i}$ is a trivial $\C^*$-bundle over $S_i$.
Thus, we get an isomorphism $\pi_1(\tilde{\Lambda}_{S_i}) \cong \pi_1(S_i) \times \pi_1(\C^*)$, by choosing a section of the bundle. From now on we will assume that this section and the base point are fixed. Let 
$\gamma_i$ be the generator of $\pi_1(\C^*)$ in $\pi_1(\tilde{\Lambda}_{S_i})$ and $\Gamma_i$ be the image of $\gamma_i$ in $\pi_1(\C^2 - S)$. 
If $A$ is a local system in $\C^2 - S$, then $A$ is a perverse sheaf in $\C^2 - S$.
In this situation, we get:

\begin{lem}(\cite[Lemma 1.1]{MV3})
\label{varsim}
We have $\psi(A)=\psi_c(A)= \tilde{A}_i$ on $\tilde{\Lambda}_{S_i}$ and the variation map $var: \tilde{A}_i \rightarrow \tilde{A_i}$ is given by $var(a)=\gamma_i(a)-a$. 
\end{lem}

From this result we see that the local system $A$ on $\C^2 - S$ gives rise to an object in the category
$Q_{\Lambda}= \mathcal C(\psi, \psi_c, var)$, 
consisting of the pairs $(\tilde{A}_i, B_i)$ (where $B_i$ is a $\pi_1(S_i)$-module), and commutative diagrams
$$\xymatrix{
\tilde{A}_i \ar[rr]^{var} \ar[rd] & & \tilde{A}_i \\
 & B_i \ar[ru] &
}$$
$i=1, \ldots, n$.

All perverse sheaves can be described in this way:

\begin{prop}{\cite[Proposition 1.2.]{MV3}}
\label{TeorP1}
The category $\mathcal M(\C^2- \{ 0 \})$ is equivalent to $Q_{\Lambda}$. 
\end{prop}

We know that the monodromy of the stalk at a fix point $x_0\in \C^2-S$ defines a functor from rank 1 local systems on $\C^2-S$ to finite dimensional representations of $\pi_1(\C^2-S)$ on the vector space $\C e$.
Now we want to apply this result.
\subsection{Locally constant sheaves on $\C\setminus S$}
A presentation of $\pi_1(\C^2-S)$ is given by $\langle  \Gamma_1, \ldots, \Gamma_n \rangle / R$, where $R$ is the subgroup of the free group $\langle  \Gamma_1, \ldots, \Gamma_n \rangle$ generated by the relations 
$\Gamma_1 \Gamma_2 \ldots \Gamma_n = \Gamma_2 \ldots \Gamma_n \Gamma_1 = \Gamma_n \Gamma_1 \ldots \Gamma_ {n-1}$. 
Observing that an action of $\pi_1(\C^2-S)$ on a vector space $\C e$ is given by $\{ \Gamma_i e = a_ie \}$, $i=1, \ldots, n$ and that then automatically $\Gamma_1 \Gamma_2 \ldots \Gamma_n e = 
\Gamma_i \Gamma_{i+1} \ldots \Gamma_{i-1} e = \Pi_{i=1}^{n} a_i e$, we see that there is a bijective correspondence between local system of rank 1 in  $\C^2-S$ and multi-indices $a=(a_i, \ldots, a_n)$ in 
$\C^n$. We denote the locally constant sheaf on $\C^2-S$ corresponding to $a$ by $\mathcal L_{a}$.

The pull-back of $\mathcal L_{a}$ to $V_i - S_i \subset \C^2 - S$ will be a locally constant sheaf of rank 1 and corresponds to a representation of $\pi_1(V - S_i)= \pi_1(S_i) \times \pi_1(\C^*)$ on a one-dimensional 
complex space, $\tilde{A}_i \cong \C e_i$. The action of $\Gamma_i \in \pi_1(\C^*)$ is given by multiplication by $a_i: \Gamma_i e_i= a_i e_i$.

Hence, by Lemma \ref{varsim}, the variation map, $var:\tilde{A}_i \rightarrow \tilde{A}_i$, is given by multiplication by $a_i -1$: $var(e_i) = \Gamma_i e_i- e_i = a_i e_i - e_i = (a_i - 1)e_i$. Thus, in our case, 
the variation is either an isomorphism, if $a_i \neq 1$, or $0$, if $a_i=1$.

\vspace{0.2cm}

In particular, considering the inclusions 
$$\C^2 - S \xrightarrow{j^1} \C^2- \{ 0\} \xrightarrow{j^2} \C^2$$
and applying the results of Proposition \ref{3functors} we know that $^pj_*^1 \tilde{A}_i$ can be represented as the set of diagram of the following form:
\begin{equation}
\label{diagABC}
\xymatrix{
\tilde{A}_i=\C \ar[rr]^{a_i-1} \ar[rd]_{p_i=a_i-1} & & \tilde{A}_i=\C \\
 & B_i=\C \ar[ru]_{q_i=id} &
}
\end{equation}
Note that $B_i$ is determined by the condition that the diagram corresponds to a direct image $^pj_*^1$, since the map $B_i \rightarrow \tilde{A}_i$ is then the identity.   Now we can prove our first result.

\begin{prop}
\label{j1irre}
$^pj^1_*\mathcal L_{a}$ is irreducible in $\mathcal M(\C^2 - \{ 0\})$ if, and only if, $a_i \neq 1$, for all $i= 1, \ldots, n$.
\end{prop}

This result may be easily modified to give a proof of Proposition \ref{prop:Clc}.

\begin{proof}
By Proposition \ref{TeorP1}, $^pj^1_*\mathcal L_{a}$ corresponds to the set of all diagrams \eqref{diagABC} and, by Propostion \ref{irpor}, these diagrams represent irreducible objects exactly when $a_i - 1 \neq 0$, 
for all $i= 1, \ldots, n$.
\end{proof}

Now we want to calculate $^pj_*\mathcal L_{a}= {^pj_*^2(^pj^1_* \mathcal L_{a})}$.
We need the following result from \cite{MV3}, that we have modified to our situation ($m=n$, and a rank 1 bundle).

\begin{lem}{\cite[Lemma 2.2. and Lemma 2.3.]{MV3}}
\label{2em1}
Given an object $P^\bullet$ in $\mathcal M(\C^2 - \{ 0 \})$, described by a set of diagrams \eqref{diagABC} (so corresponding to a sheaf $\mathcal L_a$), we have:
\begin{itemize}
\item $\psi(P^\bullet)=cok(A \xrightarrow{(p_1, \ldots, p_n)} B_1 \oplus \ldots \oplus B_n)$;
\item $\psi_c(P^\bullet)=ker(B_1 \oplus \ldots \oplus B_n \xrightarrow{(q_1, \ldots, q_n)} A)$.
\end{itemize}
\label{var}
The map $var: \psi(P^\bullet) \rightarrow \psi_c(P^\bullet)$ is given on $B_l=\C e_l$ by 
$$var(e_l) = (- \Sigma_{i=l+1}^{l+n}p_i a_{i-1} \ldots a_{l+1}e_i) + \Pi_{i=1}^n a_ie_l - e_l.$$ 
\end{lem}

Here $p_i=a_i-1$. To read this formulae, note first that all the indices are considered as integers modulo $n$.
Then, which is not explicit in \cite{MV3}, we have to make two conventions.
First, if the sequence $a_{i-1} \ldots a_{l+1}$ is an increasing sequence (in indices, where this happens for $i=l+1$) we consider that it is equal to $1$.
Second,  for the same sequence,  if we have only two elements such that $a_{i-1} = a_{l+1}$(that is $i=l+2$), then we let $a_{i-1}...a_{l+1} = a_{l+1}$.

The map
\begin{equation}
\label{eq:varmatrix}
\oplus B_i = \C^n \rightarrow \psi(\mathcal L_{a}) \xrightarrow{var} \psi_c(\mathcal L_{a}) \rightarrow \C^n = \oplus B_i
\end{equation}
will then be given by the matrix $M_n=(d_{rc})$, ($r,c=1, \ldots, n$) where:

\begin{equation}
\label{formag}
d_{rc} = 
\left\{\begin{array}{l}

         (1-a_r)\Pi_{i=1}^{r-1}a_i\Pi_{i=c+1}^na_i, \quad \mbox{if} \quad r<c \\
         -1+ \Pi_{i=1}^{n}a_i/a_c, \quad \mbox{if} \quad r=c \\
         (1-a_r)\Pi_{i=c+1}^{r-1}a_i, \quad \mbox{if} \quad  r>c
\end{array} \right. 
\end{equation}
(Here we have reverted to standard use of indices and products.)
\begin{example}
For $n=4$, the matrix $M_{4}$ is given by 
\[ \left[ \begin{array}{cccc}
-1 + a_2a_3a_4 & (1-a_1)a_3a_4 & (1-a_1)a_4 & 1-a_1 \\
1-a_2 & -1+a_1a_3a_4 & (1-a_2)a_1a_4 & (1-a_2)a_1 \\
(1-a_3)a_2 & 1-a_3 & -1 + a_1a_2a_4 & (1-a_3)a_1a_2\\ 
(1-a_4)a_2a_3 & (1-a_4)a_3 & 1-a_4 & -1 + a_1a_2a_3
\end{array} \right]\]
\end{example}

\section{Decomposition}

\subsection{Irreducibility}

We will now prove Theorem \ref{j2irre}. There is a homotopy equivalence of stratified spaces from the case of $n$ arbitrary distinct lines through the origin to the situation of $S$ as in Section \ref{section:plane1}. Hence we may assume that we are in that situation.

The computational part is the following lemma. 

\begin{lem}
\label{propdetm} Assume that $n\geq 2.$
Let $M'_{n-1}=M'_{n-1}(a_1,...,a_n)$ be the matrix obtained from $M_n$ in (\ref{formag}), by deleting the first column and the last row. Then, 
$$\det(M'_{n-1})=(-1)^{n-1} (-1 + a_1)(-1 + \Pi_{i=1}^n a_i)^{n-2}.
$$
\end{lem}

\begin{proof}[Proof of Lemma \ref{propdetm}]
We will do induction on $n$.  For $n=2$ the statement is a short calculation.
Let $n>2$ and compare the last two columns of the matrix $M_{n+1}$. Excluding the elements in the last two rows, columns $n$ and $n+1$ only differ by a multiplicative factor (the factor $a_{n+1}$ in column $n$). 
This means that, in the computation of $\det(M'_{n})$, if we expand along the last row, all the minors, with the exception of the last two $D_{n,n-1}$ and $D_{n,n}$, are going to be zero. Note that the corresponding elements $d_{n, n-1}=-1+a_1\cdots a_{n-1}a_{n+1}$ and $d_{n, n}=(1-a_{n})a_1\cdots a_{n-1}$.

Comparing the matrix $M'_{n-1}$ and the matrix $M'_{n}$  we can see that:
\begin{itemize}
\item the last column of $M'_{n-1}$ is identical to the last column of $M'_n$ (excluding the elements in the last rows);
\item for entries $d_{rc}$ in $M'_{n-1}$, such that $r \leq c+1$, the last factor is $a_n$, while in $M'_{n}$ it is $a_{n+1}$ (with the exception of the last column, mentioned above);
\item for the other entries there are no differences.
\end{itemize}

Now set $\tilde{a}_n=a_n a_{n+1}$.
We can proceed to the computation of $M'_{n}$. We have seen that we have only two minors to consider, associated to the following entries:
\begin{itemize}
\item $d_{n, n-1}$: we observe that (considering the definition of $\tilde{a}_n$ above) the minor $D_{n,n-1}$ has the same entries as $M'_{n-1}(a_1,...,\tilde a_n)$, therefore this minor is covered by the induction hypothesis;  
\item $d_{n,n}$: again we see that the the minor $D_{n,n}$ obtained differs from $M'_{n-1}(a_1,...,\tilde a_n)$ only in its last column, where all the entries are multiplied by $a_{n+1}$, 
therefore its determinant is given by $a_{n+1}\det(M'_{n-1}(a_1,...,\tilde a_n))$.
\end{itemize}

In conclusion, noting that by the induction hypothesis,
$\det(M'_{n-1}(a_1,...,\tilde a_n))= (-1)^{n-1}(-1+a_1)(-1 + \Pi_{i=1}^{n+1} a_i)^{n-2}$, we find that
\begin{eqnarray*}
&&\det(M'_{n})= (-1)^{n+n-1} d_{n,n-1} \det(M'_{n-1})    + (-1)^{n+n} d_{n,n} a_{n+1} \det(M'_{n-1})\\
&=&((-1)^{n}(-1 + \Pi_{i=1}^{n+1} a_i)\det(M'_{n-1}) \\
&=&
 (-1)^{n}(-1+a_1)(-1 + \Pi_{i=1}^{n+1} a_i)^{n-1}.
\end{eqnarray*}

\end{proof}

\begin{proof}[Proof of Theorem \ref{j2irre}]
Assume that the numerical conditions of the theorem hold, so that in particular $a_i\neq -1,\ i=1,...,n$ . Then, by Proposition \ref{j1irre}, $P^\bullet={^pj_*^1\mathcal L_a}$ is irreducible, and described by the set of diagrams \eqref{diagABC}. 
Hence, according to 
Lemma \ref{2em1}, ${^pj_ *^2 P^\bullet}$ is described by the diagram 
\begin{equation}
\label{eq:varproof}
\xymatrix{
\psi(P^\bullet) \ar[rr]^{var} \ar[rd] & &  \psi_c(P^\bullet)\\
 & \Phi(P^\bullet) \ar[ru]_{=} &
}
\end{equation}

Since also, by assumption,  $\Pi_{i=1}^n a_i\neq -1$,  the preceding Lemma \ref{propdetm} implies that the rank of $M'_{n-1}$ is $n-1$, and by (\ref{eq:varmatrix}) and Lemma \ref{2em1}, this implies that $var$ is an isomorphism. 
By Proposition \ref{irpor}, the diagram $(\ref{eq:varproof})$ is irreducible, since it is of the form $\hat{T}(L)$. Hence we have proved that $^pj_ *^2({} ^p{j_*}^1\mathcal L_a)={}^pj_*\mathcal L_a$ is irreducible. But $j$ is an affine morphism between affine varieties and so $Rj_*$ is t-exact and consequently $^pj_* = Rj_*$ (see \cite [Corollary 5.2.17] {BBD}).

Conversely, assume that ${^pj_* \mathcal L_a}$ is irreducible. If, say $a_1=1$, then $^pj_*^1 \mathcal L_a$ is not irreducible by Proposition \ref{j1irre}, and, since $^pj_*^2$ is t-exact (see \cite[Proposition 1.4.16]{BBD}), 
not $^pj_*\mathcal L_a$ either. Hence $a_i \neq 1$, $i=1, \ldots, n$. If then $\Pi_{i=1}^n a_i = 1$, $var$ is not an isomorphism, and again, by Proposition \ref{irpor}, ${^pj_* \mathcal L_a}$ is not irreducible (note that 
we know by Lemma \ref{2em1} that $\psi(P^\bullet)$ is non-zero).

\end{proof}

\subsection{Decomposition Factors}

If the irreducibility conditions for $^pj_*^1 \mathcal L_a$ and $Rj_*\mathcal L_a$ are not satisfied, we can still ask for the number of decomposition factors of these objects.
Let $c(P^\bullet)$ represent the number of decomposition factors of the perverse sheaf $P^\bullet$.

\begin{thm}
\label{teopr}
Assume that $a_1, \ldots, a_k = 1$ and $a_{k+1}, \ldots, a_n \neq 1$.
\begin{itemize}
\item If $k=n$, then $c(Rj_*\mathcal L_a)= 2n$.
 \item If $k<n$ and $\Pi_{i=1}^n a_i = 1$, then $c(Rj_*\mathcal L_a)= n+k-1$.
 \item If $k<n$ and $\Pi_{i=1}^n a_i \neq 1$, then $c(Rj_*\mathcal L_a)=k+1$.
\end{itemize}
\end{thm}

As in the previous theorem we will use that $Rj_*\mathcal L_a={}^pj_ *^2({}^p{j_*}^1\mathcal L_a)$.

\subsubsection{Decomposition of ${}^p{j_*}^1\mathcal L_a$}
By Proposition \ref{TeorP1},  $^pj_*^1 \mathcal L_a$ can be represented by the following diagram:

$$
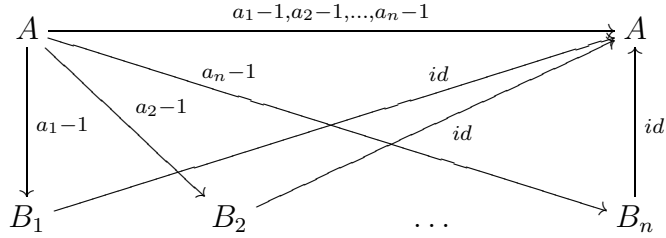
\xymatrix{
A \ar[dd]^{a_1-1} \ar[ddrr]^{a_2-1} \ar[ddrrrrrr]^<<<<<<<<<<<<<<<<<<<<{a_n -1} \ar[rrrrrr]^{a_1-1,a_2-1,\ldots, a_n -1}& & & & & & A \\ 
 & & & & & & \\
B_1 \ar[uurrrrrr]^>>>>>>>>>>>>>>>>>>>>{id} & 
 & B_2 \ar[uurrrr]_>>>>>>>>>>>>>>>>>>>>{id} & 
& \ldots & 
 & B_n \ar[uu]_{id} } 
$$
Here $A, B_i,\ i=1,...,n$ are copies of $\C^1$, and the uppermost arrow from $A$ to $A$ represents $n$ arrows, making $n$ commutative triangles.
\begin{lem}
If $a_1, \ldots, a_k = 1$, according to Proposition \ref{j1irre}, $^pj_*^1 \mathcal L_a$ is not irreducible and can be decomposed into $D_I$ and $D_{II}$,  such that 
$$ 
D_I \rightarrow {^pj_*^1 \mathcal L_a} \rightarrow D_{II}
$$
is a short exact sequence.

$D_I$ and $D_{II}$ are represented, respectively, by the diagrams
$$
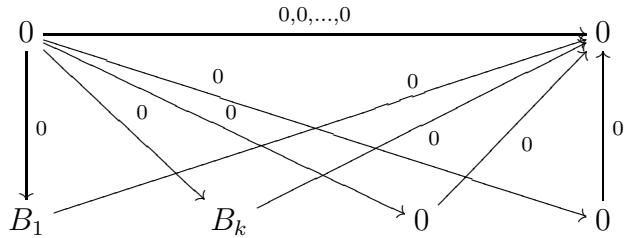
\xymatrix{
0 \ar[dd]^0 \ar[ddrr]^0 \ar[ddrrrr]^0 \ar[ddrrrrrr]^<<<<<<<<<<<<<<<<<<<<0 \ar[rrrrrr]^{0,0, \ldots, 0} & & & & & & 0 \\ 
 & & & & & & & & \\
B_1 \ar[uurrrrrr]^>>>>>>>>>>>>>>>>>>>>0 & 
& B_k \ar[uurrrr]_>>>>>>>>>>>>>>>>>>>>0 &
& 0 \ar[uurr]_0 & 
& 0 \ar[uu]_0 } 
$$

and

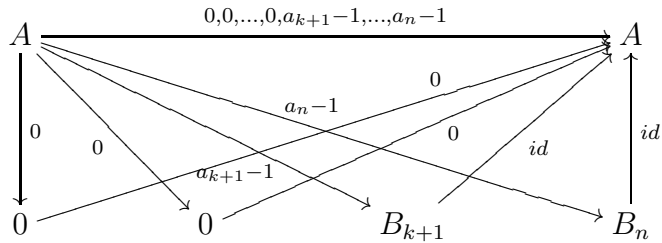
$$\xymatrix{
A \ar[dd]^0 \ar[ddrr]_0 \ar[ddrrrr]_<<<<<<<<<<<<<<<<<<<<<<<<<<<{a_{k+1}-1} \ar[ddrrrrrr]^<<<<<<<<<<<<<<<<<<<<<<<<<<<<<<{a_n-1} \ar[rrrrrr]^{0,0, \ldots, 0 , a_{k+1}-1, \ldots, a_n -1} & & & & & & A \\ 
 & & & & & & & & \\
0 \ar[uurrrrrr]^>>>>>>>>>>>>>>>>>>>>0 &
& 0 \ar[uurrrr]_>>>>>>>>>>>>>>>>>>>>0 &
& B_{k+1} \ar[uurr]_{id} &
& B_n \ar[uu]_{id} } 
$$
\end{lem}
\begin{proof} The property of the diagram category that injections are exactly the maps that are injections on each of the vertices, and similarily for surjections, makes this clear.
\end{proof}

We say that the diagrams of the form that represent $D_I$ are of type $I$, characterized by the fact that the two upper vertices are $0$. Those of the form that represent $D_{II}$, where  all diagonal maps are isomorphisms, are said to be of type $II$.
 

\subsubsection{Decomposition of ${}^pj_*^2D_I$}
\label{subsubsection: dec1}  We can use the description of how the functor $^pj^2_*$ is described in the diagram category, in Lemma \ref{2em1} to describe what it does to the above diagrams.
\begin{lem}
Apply $^pj^2_*$ to $D_I$. Then:
\begin{itemize}
 \item $\psi(D_I): coker (0 \xrightarrow{0, \ldots , 0} B_1 \oplus \ldots \oplus B_k \oplus 0 \oplus \ldots \oplus 0) = \C^k = \langle e_1, \ldots , e_k \rangle$ ;
 \item $\psi_c(D_I): ker(B_1 \oplus \ldots \oplus B_k \oplus 0 \oplus \ldots \oplus 0 \xrightarrow{0, \ldots , 0} 0) = \C^k = \langle e_1, \ldots , e_k \rangle$:
 \item $var_I(e_l)= \Pi_{i=1}^n a_i e_l - e_l$ (since $p_i = a_i -1 = 0$, $i=1, \ldots, k$).
\end{itemize}
\begin{equation}
\label{diadi}
\xymatrix{
^pj_*^2(D_I): & \C^k \ar[rr]^{var_I} \ar[rd]_{var_I} & & \C^k \\ 
& & \C^k \ar[ur]_{id} &  }
\end{equation}
\end{lem}

It is clear that $^pj_*^2(D_I)$ is not irreducible.

\begin{lem}
\label{lemdi}
\begin{itemize}
\item If $\Pi_{i=1}^n a_i = 1$, then $var_I=0$ and $c(^pj_*^2(D_I))=2k$.
\item If $\Pi_{i=1}^n a_i \neq 1$, then $var_I$ is represented by the order $k$ matrix
\begin{equation}
\label{matrdi}
\left[ \begin{array}{cccc}
-1 + \Pi_{i=1}^n a_i & 0 & \dots & 0 \\
0 & -1 + \Pi_{i=1}^n a_i & \ldots & 0 \\
\vdots & \vdots & \vdots & \vdots \\ 
0 & 0 & \ldots & -1 + \Pi_{i=1}^n a_i
\end{array} \right] 
\end{equation}

and $c(^pj_*^2(D_I))=k$.
\end{itemize}
\end{lem}

\begin{proof}
\begin{itemize}
\item When $\Pi_{i=1}^n a_i = 1$, clearly from the formula for $var_I(b_le_l)$ we have that $var_I=0$ and \eqref{diadi} is an extension of two factors of the following form:
$$\xymatrix{
\mathcal{D_I'}: & 0 \ar[rr]^{0} \ar[rd]_0 & & 0 \\ 
& & \C^k \ar[ur] &  }
\quad \mbox{and} \quad
\xymatrix{
\mathcal{D_I''}: & \C^k \ar[rr]^{0} \ar[rd]_0 & & \C^k \\ 
& & 0 \ar[ur] &  }$$

$\mathcal{D_I'}$ can be decomposed into $k$ objects of the form

$$\xymatrix{
0 \ar[rr]^{0} \ar[rd]_0 & & 0 \\ 
& \C \ar[ur] &  }$$

$\mathcal{D_I''}$ can be decomposed into $k$ objects of the form

$$\xymatrix{
\C \ar[rr]^{0} \ar[rd]_0 & & \C \\ 
& 0 \ar[ur] &  }$$

We have, in total, \underline{$2k$ decomposition factors}. 
It should be noted that this proves the theorem for the case $k=n$.

\item When $\Pi_{i=1}^n a_i \neq 1$, $var_I$ is given by \eqref{matrdi} and there are $k$ decomposition factors of \eqref{diadi}, 
all of the form:

$$\xymatrix{
\C \ar[rr]^{\Pi_{i=1}^n a_ie_1-e_1} \ar[rd] & & \C \\ 
& \C \ar[ur] &  }.
$$




We have, in total \underline{$k$ decomposition factors}.
\end{itemize}
\end{proof}
\subsubsection{Decomposition of ${}^p{j_*}^2D_{II}$} 

Assume now that $k<n$, so that $D_{II}$ is non-trivial. We use the description of the functor $^pj^2_*$ in Lemma \ref{2em1} again.
\begin{lem}
Apply $^pj^2_*$ to $D_{II}$. Then:
\begin{itemize}
 \item $\psi(D_{II}): coker (A \xrightarrow{0, \ldots , 0, a_{k+1} -1 , \ldots a_n -1} 0 \oplus \ldots \oplus 0 \oplus B_{k+1} \oplus \ldots \oplus B_n) = \C^{n-k-1} = \langle e_{k+1}, \ldots , e_{n-1} \rangle$ ;
 \item $\psi_c(D_{II}): ker(0 \oplus \ldots \oplus 0 \oplus B_{k+1} \oplus \ldots \oplus B_n \xrightarrow{0, \ldots , 0, a_{k+1} -1 , \ldots a_n -1} A) = \C^{n-k-1} = \langle e_{k+1}, \ldots , e_{n-1} \rangle$:
 \item $var_{II}(e_l): (- \Sigma_{i=l+1}^{l+n}p_i a_{i-1} \ldots a_{l+1}e_i) + \Pi_{i=1}^n a_ie_l - e_l$, for $l=k+1, \ldots, n-1$. 
\end{itemize}
$$\xymatrix{
^pj^2_*(D_{II}): & \C^{n-k-1} \ar[rr]^{var_{II}} \ar[rd]_{var_{II}} & & \C^{n-k-1} \\ 
& & \C^{n-k-1} \ar[ur]_{id} &  } 
$$
\end{lem}
Note that there is a choice of basis in the lemma. It follows that we have that $^pj^2_*(D_{II})$ is an extension of the following two diagrams:
$$\xymatrix{
\mathcal{D_{II}'} : & 0 \ar[rr]^{0} \ar [rd] & & 0 \\
 & & \C^{n-k-1}/im(var_{II}) \ar[ur] & 
}$$ and
$$\xymatrix{
\mathcal{D_{II}''}: & \C^{n-k-1} \ar[rr]^{var_{II}} \ar[rd] & & \C^{n-k-1}\\
 & & im(var_{II}) \ar[ur] &
}.$$

\begin{lem}
\label{lemdii}
Assume that $a_1, \ldots, a_k = 1$, $a_{k+1}, \ldots, a_n \neq 1$ and $\Pi_{i=1}^n a_i=1$. Then, $c(\mathcal{D'_{II}})=n-k-2$ and $c(\mathcal{D''_{II}})=1$. Furthermore, if $\Pi_{i=1}^n a_i\neq 1$, then $^pj^2_*(D_{II})$ is irreducible.
\end{lem}

In order to prove this lemma we need the following result.

\begin{lem}
Assume that $\Pi_{i=1}^n a_i=1$, and let $A$ be the matrix that represents $var_{II}$. Then $rank(A)=1$.
\end{lem}

\begin{proof}
Set $\beta_k^n = \Pi_{i=k}^n a_i$.
Recall the description \eqref{formag} of the elements of the matrix that represents $var$.
The matrix $A$ has $n-k-1$ rows and columns.
We are going to apply the Gauss elimination to the rows of $A$ in order to obtain its rank. For the matrix $A$, we have:
\begin{itemize}
 \item $d_{11} = -1+\beta_{k+2}^n$;
 \item $d_{1j} = (1-a_{k+1}) \beta_{c+k+1}^n$.
\end{itemize}

Since we assumed $a_1, \ldots, a_k = 1$, $a_{k+1}, \ldots, a_n \neq 1$ and $\beta_{1}^n=1$, we observe that $\beta_{k+1}^n=1$. We can conclude that $\beta_{k+2}^n\neq 1$ and that 
$a_{k+1}=1/\beta_{k+2}^n$. 
We proceed in the following way:
\begin{enumerate}
 \item Multiply the first row by $1/(-1+\beta_{k+2}^n)$. Then $d_{11}=1$ and
 $$d_{1j}= \frac{(1-a_{k+1})\beta_{c+k+1}^n}{-1 + \beta_{k+2}^n} = \frac{(1 - 1/\beta_{k+2}^n) \beta_{c+k+1}^n}{-1 + \beta_{k+2}^n} = $$ 
 $$= \left( \frac{\beta_{k+2}^n -1 }{\beta_{k+2}^n} \right)  \beta_{c+k+1}^n \frac{1}{\beta_{k+2}^n -1} = \frac{1}{\beta_{k+2}^{c+k}}$$
 \item Now we want to cancel all the elements in the first column of $A$. In each step, multiply the first row by $(a_{r+k}-1)\beta_{k+2}^{r+k-1}$, for $r= 1, \ldots, n-k-1$, and sum it to the correspondent row. 
We have 3 possible results:
 \begin{itemize}
  \item if we sum to an element $d_{rc}$, such that $r<c$, given by \\
  $(1-a_{r+k})\beta_{k+1}^{r+k-1}\beta_{c+k+1}^{n}$, the result is:
  $$\frac{1}{\beta_{k+2}^{c+k}} (a_{r+k}-1)\beta_{k+2}^{r+k-1} + (1-a_{r+k})\beta_{k+1}^{r+k-1}\beta_{c+k+1}^n =$$
  $$ = (\beta_{k+2}^{r+k-1})(1-a_{r+k}) \left( a_{k+1} \beta_{c+k+1}^n - \frac{1}{\beta_{k+2}^{c+k}} \right) =$$ 
  $$ = (\beta_{k+2}^{r+k-1})(1-a_{r+k}) \left[ \left(\beta_{k+1}^n - 1 \right) \frac{1}{\beta_{k+2}^{c+k}} \right] = 0$$
  \item if we sum to an element $d_{rc}$, such that $r=c$, given by 
  $-1 + (\beta_{k+1}^n /a_{c+k})$, the result is:
  $$\frac{1}{\beta_{k+2}^{c+k}} (a_{c+k}-1)\beta_{k+2}^{c+k-1} + (-1+\beta_{k+1}^{n}/a_{c+k}) =$$ 
  $$= \frac{1}{a_{c+k}} (a_{c+k}-1) -1 + \frac{1}{\beta_{k+2}^{n}}\beta_{k+2}^{n}/a_{c+k} =$$ 
  $$= \frac{1}{a_{c+k}} (a_{c+k}-1) -1 + \frac{1}{a_{c+k}} = 0$$
  \item if we sum to an element $d_{rc}$, such that $r>c$, given by 
  $(1-a_{r+k})\beta_{c+k+1}^{r+k-1}$, the result is:
  $$\frac{1}{\beta_{k+2}^{c+k}} (a_{r+k}-1)\beta_{k+2}^{r+k-1} + (1-a_{r+k})\beta_{c+k+1}^{r+k-1}=$$
  $$ = (a_{r+k} - 1) \beta_{c+k+1}^{r+k-1} \left( \frac{1}{\beta_{k+2}^{c+k}} \beta_{k+2}^{c+k} - 1 \right) = 0$$ 
 \end{itemize}
\end{enumerate}

We observe that, except for the first row, all the elements of the matrix $A$ are canceled. Therefore $rank(A)=1$. 
\end{proof}

\begin{proof}(of Lemma \ref{lemdii})
We observe that, according to the irreducibility criterion, the diagram $\mathcal{D_{II}''}$ is irreducible.

The object represented by diagram $\mathcal{D_{II}'}$ has $dim(\C^{n-k-1}/im(var_{II}))$ decomposition factors.
Since $rank A = dim(im(var_{II}))=1$, we conclude that $dim(\C^{n-k-1}/im(var_{II}))=n-k-2$. Hence, the  diagram $\mathcal{D_{II}'}$ can be decomposed into $n-k-2$ factors.

Finally, if $\Pi_{i=1}^n a_i \neq 1$, then a determinant computation entirely similar to the one in Lemma \ref{propdetm} shows that $var_{II}$ is invertible 
and hence that
$^pj^2_*(D_{II})=\mathcal{D''_{II}}$ is irreducible. 
\end{proof}


\begin{proof}(of Theorem \ref{teopr})
The case when $k=n$ was treated in \ref{subsubsection: dec1} above. From
Lemma \ref{lemdi} and Lemma \ref{lemdii},  we get 
\begin{itemize}
\item If $\Pi_{i=1}^n a_i = 1$ then $c(Rj_*\mathcal L_a)= 2k + (n-k-2) +1 = n+k-1$.
\item If $\Pi_{i=1}^n a_i \neq 1$ then 
$c(Rj_*\mathcal L_a)= k+1$.
\end{itemize}

This finishes the proof.

\end{proof}

\end{document}